\documentclass[12pt,twoside,leqno]{amsart}
\usepackage{amsmath,amscd,amssymb,amsfonts}
\usepackage{graphicx}
\newcommand{\A}{\mathbb{A}}

\newcommand{\Z}{\mathbb{Z}}
\newcommand{\Q}{\mathbb{Q}}

\newcommand{\C}{\mathbb{C}}

\renewcommand{\P}{\mathbb{P}}

\newcommand{\AS}{\mathcal{A}}

\newcommand{\DS}{\mathcal{D}}
\newcommand{\FS}{\mathcal{F}}
\newcommand{\RS}{\mathcal{R}}
\newcommand{\ZS}{\mathcal{Z}}

\newcommand{\lra}{\longrightarrow}

\DeclareMathOperator{\Pic}{Pic}
\DeclareMathOperator{\NS}{NS}
\DeclareMathOperator{\Aut}{Aut}

\theoremstyle{plain}
\newtheorem{theorem}{Theorem}[section]
\newtheorem{corollary}[theorem]{Corollary}
\newtheorem{proposition}[theorem]{Proposition}
\newtheorem{lemma}[theorem]{Lemma}

\theoremstyle{definition}

\newtheorem{remark}[theorem]{Remark}

\title{Bhabha Scattering and a special pencil of K3~surfaces}
\author{Dino Festi and Duco van Straten}
\begin{document}

\begin{abstract}
We study a pencil of K3 surfaces that appeared in the $2$-loop diagrams in 
Bhabha scattering. By analysing in detail the Picard lattice of the general 
and special members of the pencil, we identify the pencil with the 
celebrated Ap\'ery--Fermi pencil, that was related to Ap\'ery's proof of the 
irrationality of  $\zeta(3)$ through the work of F. Beukers, C. Peters and 
J. Stienstra. The same pencil appears miraculously in different and seemingly 
unrelated physical contexts.
\end{abstract}
\maketitle
\section*{Introduction}
The electron-positron scattering process $e^++e^-\rightarrow e^++e^-$ is called  {\em Bhabha scattering}, after Homi J. Bhabha, who calculated
the differential cross-section to lowest order in $1935$, \cite{bhabha}. 
This calculcation can be found in almost any textbook on 
quantum field theory and is now routinely relegated to the exercise 
sheets of the courses on the subject.

At lowest order, there are two Feynman diagrams contributing to the amplitude 
of this process, called the scattering and the annihilation diagram.\\

\begin{center}
\includegraphics[height=4.5cm]{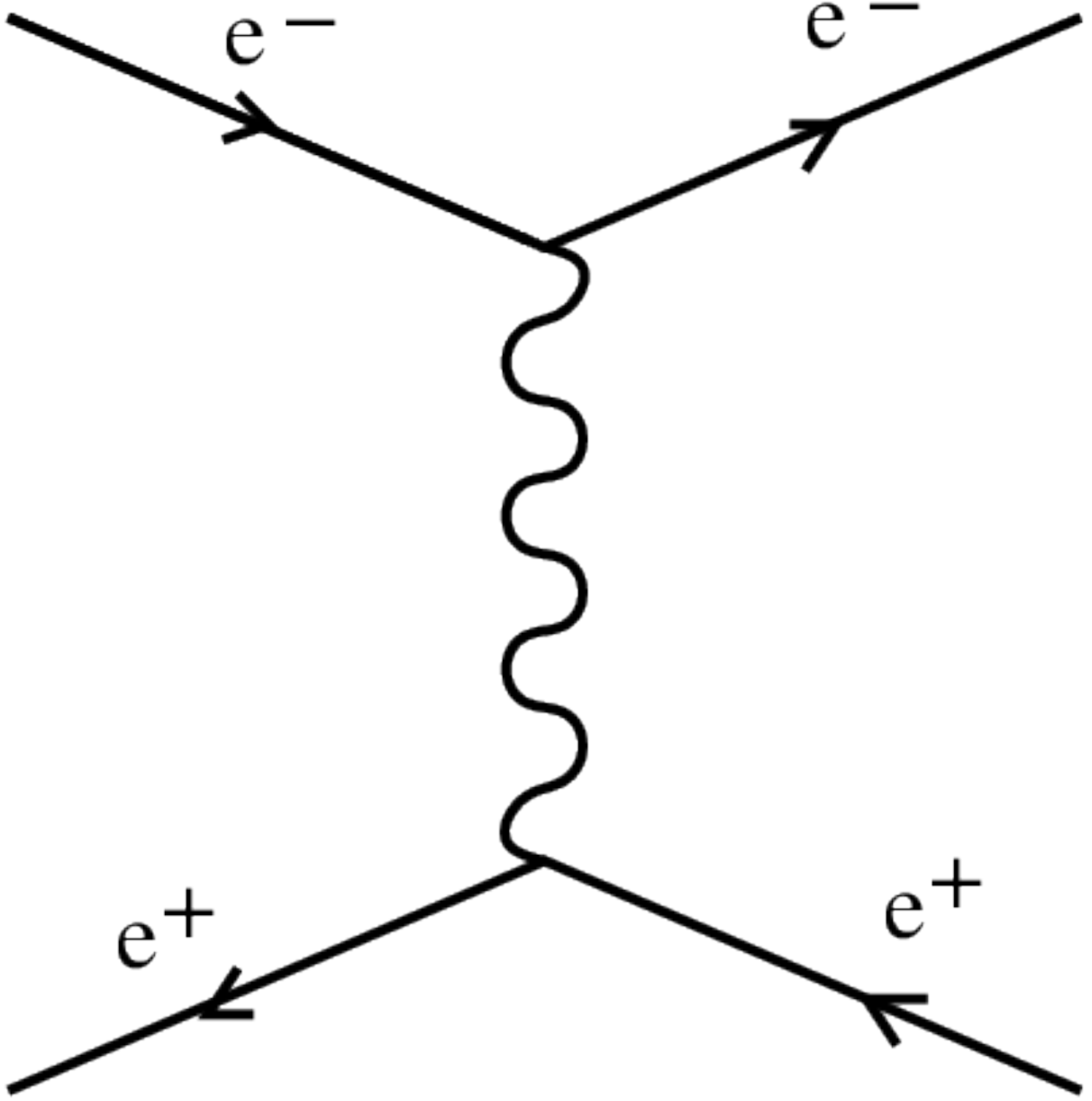}
\hspace{2.4cm}
\includegraphics[height=4.5cm]{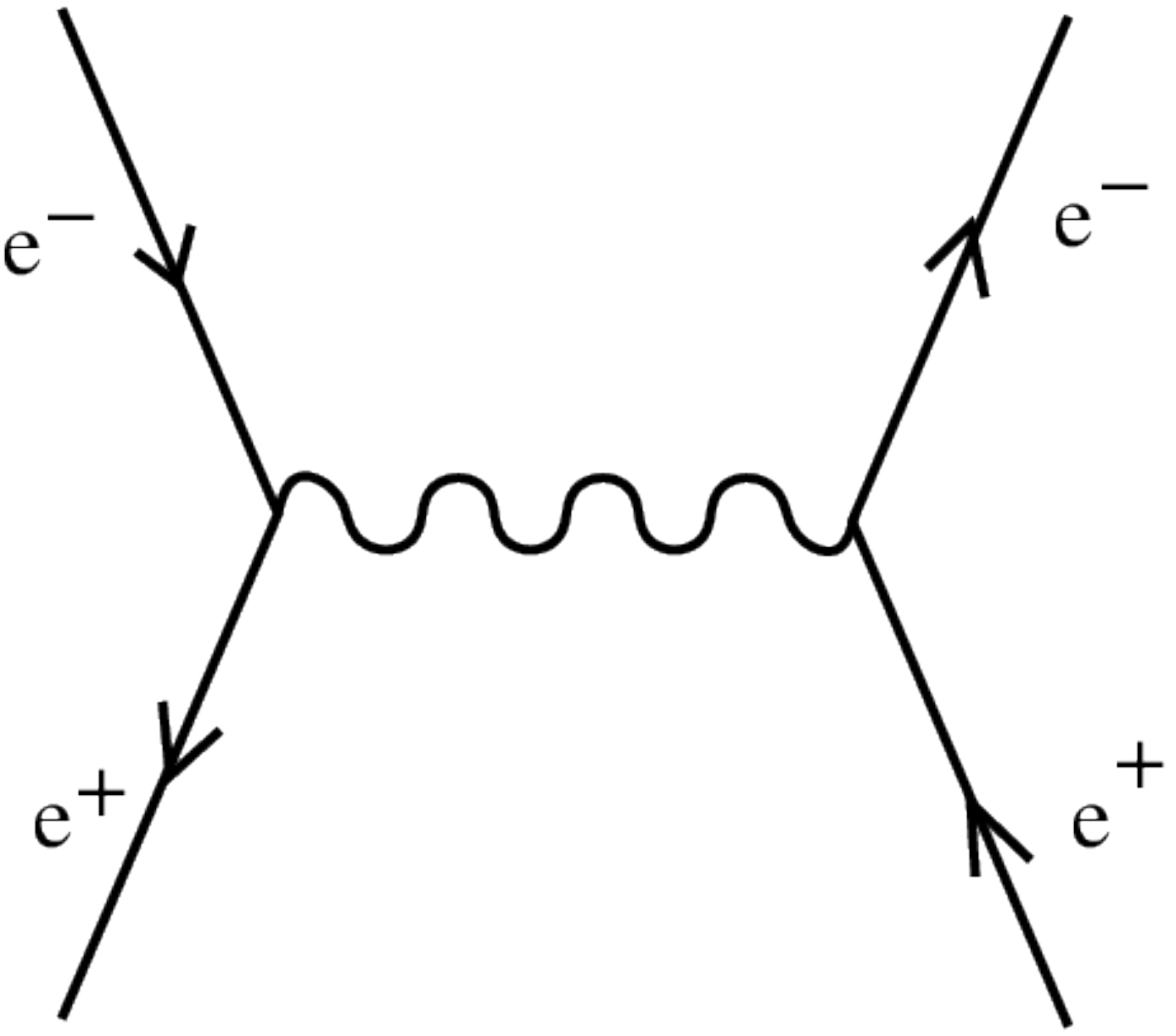}
\vskip 10pt
{\em Scattering and annihilation diagrams.}
\end{center}

In terms of the {\em Mandelstam variables}
\[ s=-(p_1+p_2)^2, \;\;\;\;t=-(p_1-p_3)^2\]
the answer obtained by Bhabha can be written in the form 
\[\frac{ d\sigma}{d\Omega} =
\frac{\alpha^2}{s}\left\{ \frac{1}{s^2}\left[ st+\frac{s^2}{2}+(t-2m^2)^2\right]
+
\frac{1}{st}\left[ (s+t)^2-4m^4\right]
+
\frac{1}{t^2}\left[ st+\frac{t^2}{2}+(s-2m^2)^2\right]
\right\}
\]
where $\alpha=e^2/\hbar c \approx 1/137$ is the fine-structure constant and $m$ is the mass of the electron (cf. \cite[pp. 2, 3]{boncianietall2}).
 
The three terms in the formula have an interesting interpretation in terms
of the two Feynman diagrams. As the determination of the cross section involves 
the square modulus of the amplitude, the first and last term in the formula
describe the contributions  of scattering and annihilation, whereas the middle term represents the interference (or `exchange') between the two virtual processes.

According to the rules of Quantum Electrodynamics (QED), 
the complete amplitude for Bhabha scattering
appears as a sum of integrals that correspond to certain Feynman diagrams. 
The next order of approximation depends on the virtual processes described by
the ten different $1$-loop diagrams and at $2$-loops there are many diagrams 
to consider.

The corresponding integrals in fact diverge and need to be regularised. One 
of the most powerful regularisation schemes is {\em dimensional regularisation}.  All integrals are taken in $D$ spacetime dimensions, then developed as a
Laurent series in $\epsilon=\frac{4-D}{2}$ and regularised by subtracting the 
polar part. 
It was found in \cite{boncianietall}, \cite{boncianietall2} that 
all coefficients in the $\epsilon$-expansion of $1$-loop Feynman integrals can be represented in 
terms of so-called {\em harmonic polylogarithms}, a class of functions introduced in \cite{remiddivermaseren}, 
evaluated at arguments that are {\em rational expressions} in variables $x,y$
related to the Mandelstam variables
\[ -s=m^2 \frac{(1-x)^2}{x},\;\; x=\frac{\sqrt{4m^2-s}-i\sqrt{s}}{\sqrt{4m^2-s}+i\sqrt{s}},\]
\[ -t=m^2 \frac{(1-y)^2}{y},\;\; y=\frac{\sqrt{4m^2-t}-i\sqrt{t}}{\sqrt{4m^2-t}+i\sqrt{t}}.\]

In \cite{hennsmirnov} the $2$-loop master integrals were reconsidered 
and  expressed in terms of Chen iterated integrals:
the irrationality
\[ Q=\sqrt{\frac{(x+y)(1+xy)}{x+y-4xy+x^2y+xy^2}}\]
was needed as argument in the logarithmic forms 
of the differential equation associated to one special master integral.
The same irrationality appears in~\cite{bddpt},
where the same $2$-loop master integrals are reconsidered,
but this time are expressed in terms of elliptic polylogarithms.
The question arises if one can `undo' the square root by a rational 
substitution of the form
\[ x=x(s,t),\;\;\;y=y(s,t).\]
More precisely, if we introduce a further variable $z$ and put $z=(x+y)/Q$, we
obtain, after squaring and clearing denominators, a quartic surface in the complex affine space in $\A_\C^3$ 
with equation
\begin{equation}\label{eq:Q} 
Q\colon z^2 (1+xy) =(x+y)(x+y-4xy+x^2y+xy^2). 
\end{equation}
\begin{center}
\includegraphics[height=5cm]{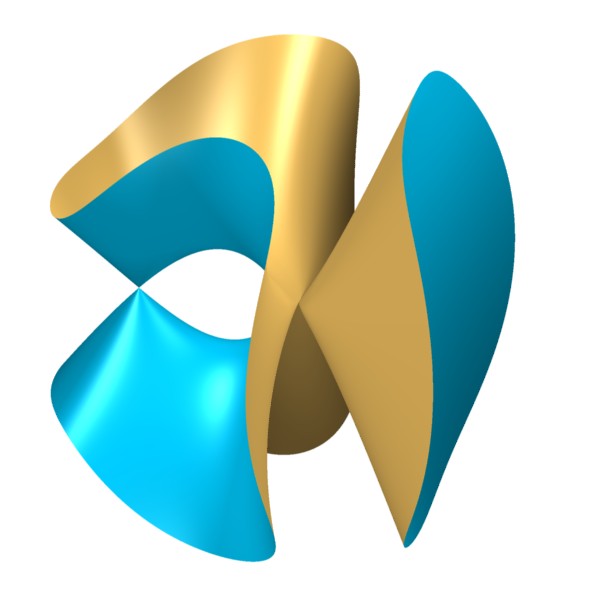}

 {\em The quartic $Q$.}
\end{center}
\begin{remark}\label{r:Qsing}
The quartic $Q$ has eight singular points, all of type $A_k$ for some $k$, 
two of which are visible in the picture above;
six more singular points lie in the plane at infinity.
More precisely, the singular points have the following coordinates.
\[
\begin{array}{|c|c|c|}
\hline
\textup{Coordinates}&\textup{Singularity}\\[1mm]
\hline
\hline
(0:0:0:1)  & A_3\\
\hline
(0:1:1:0), (0:1:-1:0), (1:0:1:0), (1:0:-1:0)   & A_2\\
\hline
(1:1:0:1), (0:0:1:0), (1:-1:0:0)  & A_1\\
\hline
\end{array}
\]
\end{remark}

As this quartic has only simple singularities (see remark above), it is birational to a 
K3 surface, which implies that  no rational parameterisation is 
possible, see also \cite{bvsw}.

Using coordinate transformations, this surface defined by the irrationality $Q$ 
can be transformed in many different forms. In his talk (June 2, 2015) at the 
MITP conference {\em Amplitudes, Motives and beyond}, Johannes Henn 
mentioned the particular nice surface with the equation
\[0=1+\frac{x^2}{1-x^2}+\frac{y^2}{1-y^2}+\frac{z^2}{1-z^2},\]
which after clearing of denominators leads to the equation 
\[B\colon 0=f(x,y,z):=1-(x^2y^2+y^2z^2+z^2x^2)+2 x^2y^2z^2.\]
He posed the question if this surface is rational, i.e., if there exist 
non-trivial rational functions $x(s,t),y(s,t),z(s,t)$ that satisfy the 
above equation identically, 
thus providing a rational parametrisation of the surface $B$. 
The answer to this question is no. 
In this paper we will compute the geometric Picard lattice of a family of K3 surfaces containing $B$ and 
we will show that this surface is in fact a well-known beauty,
as stated in  the following theorem.\\
 
\centerline{\bf Theorem} 
\vskip 10pt
{\em The surface $B$ is birational to a K3 surface with Picard number 
equal to $20$. 
Its geometric Picard lattice is isomorphic to
\[ U \oplus E_8(-1)^2 \oplus \langle-4 \rangle \oplus \langle -2\rangle.\]}

It is not hard to see that the $8$ points $(x,y,z)=(\pm 1,\pm 1,\pm 1)$
are ordinary double points and make up all singularities of the surface $B$ in 
affine space, but more complicated singularties are present at infinity.
By the reciprocal transformation 
\[  u = \frac{1}{x},\;\;\;v = \frac{1}{y},\;\;\;w = \frac{1}{z}\]
the surface is mapped to a surface $R$
\[R:\;\;\; 0=1+\frac{1}{u^2-1}+\frac{1}{v^2-1}+\frac{1}{w^2-1}\]
which after clearing denominators produces a slightly simpler equation 
\[R\colon 0=u^2v^2w^2-u^2-v^2-w^2+2 \]
for the surface.

\begin{center}
\includegraphics[height=5cm]{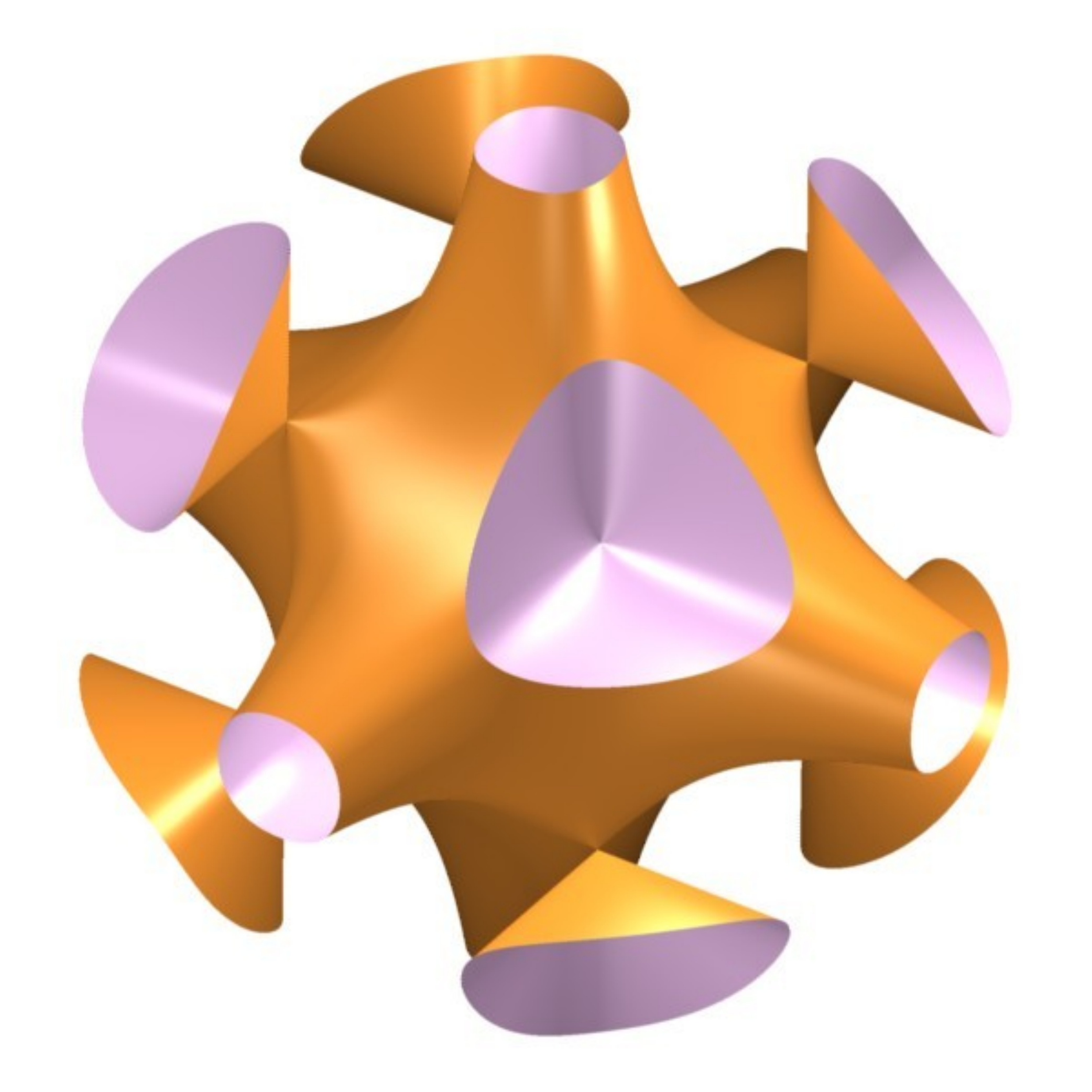}
\hspace{2cm}
\includegraphics[height=5cm]{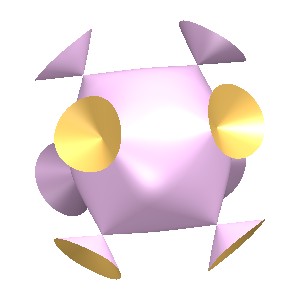}

{\em The surface $B$ and its reciprocal $R$.}
\end{center}

\section{The pencil and representation as double sextic}
The above surface $R$ can be seen as a member of a nice symmetric (cf. Remark~\ref{r:Symmetries}) one-parameter 
family of surfaces 
with equation

\[\RS: 0=1+s+\frac{1}{u^2-1}+\frac{1}{v^2-1}+\frac{1}{w^2-1}=0 \subset \A^1\times \A^3 ,\]
where we consider $s$ as parameter. We write
\[\pi:\RS \to \A^1,\;\;\;(u,v,w,s) \mapsto s\]
for the projection and denote by $\RS_s=\pi^{-1}(s)$ the fibre over $s$, so $R=\RS_0$.
Clearing denominators we can write the equation for $\RS_s \subset \A^3$ as:
\[ u^2v^2w^2-u^2-v^2-w^2+2+s(u^2-1)(v^2-1)(w^2-1)=0. \]

\begin{remark}\label{r:Symmetries}
From its defining equation, 
one can easily see that $\RS_s$ admits (at least) 48 automorphisms:
those coming from changing the sign of and permuting  the coordinates.
More precisely, these automorphisms generate a subgroup of $\Aut (\RS_s)$ isomorphic to
$$
S_3 \ltimes (\Z/2\Z)^3, 
$$
where the action of $S_3$ on $(\Z/2\Z)^3$ is given by permuting the coordinates.
\end{remark}

In order to study the geometry of these surfaces, we will brutally break the symmetry
and use this equation to express $w^2$ in terms of $u,v$ (and $s$), obtaining

\[
w^2=\frac{u^2+v^2-2+s(u^2v^2-u^2-v^2+1)}{u^2v^2 -1 +s(u^2v^2-u^2-v^2+1)}.
\]
Using the birational map $(u,v,w) \mapsto (1/u,1/v,w)=:(x,y,w)$ the threefold $\RS$ 
is seen to be  birationally equivalent to the threefold given by
\[ w^2=\frac{f_1}{f_0} \]
with
\[
f_i(x,y):=(1-x^2y^2) +(s-i)(x^2-1)(y^2-1), \;\;\;i=0,1.\]

Using the map $(x,y,w)\mapsto (x,y,wf_1(x,y))$, we see that $\RS$ is birationally equivalent to the threefold defined by the equation

\[w^2=f_1 \cdot f_0 .\]
This represents our threefold as a double cover, ramified over the union of two $s$-dependent 
quartic curves. 
Let us write $F_i(x,y,z)$ for the homogenization in $z$ of the polynomial $f_i$, for 
$i=0,1$.  The quartics in $\P^2$ with homogeneous coordinates $(x,y,z)$ defined by $F_1$ 
and $F_2$ intersect in the points $(\pm 1:\pm 1: 1)$ with multiplicity $2$ and in the points 
$(1:0:0), (0:1:0)$ with multiplicity $4$. The last two points form also the singular locus of 
both curves.

The Cremona transformation $\gamma$ on the base points 
$$(1:0:0),\;\; (0:1:0),\;\; (1:1:1)$$
transforms these singular quartics into non-singular cubics $B_i$
defined by the vanishing of the polynomials
\[G_i(x,y,z):=(x^2+y^2)z-2xy(x+y)+(s+1-i)(2x-z)(2y-z)z,\;\;i=0,1.\]
We let $\P:=\P(1,1,1,3)$ be the weighted projective space on the variables
$x,y,z$ of weight $1$ and a variable $w$ of weight $3$ and let $\DS$ to be 
the threefold in $\A^1 \times \P$ defined by the equation

\begin{equation}\label{eq:newpencil}
\DS:\;\;w^2= G_0\cdot G_1 .
\end{equation}
The fibre over $s$ of the map $\DS \to \A^1$ will be denoted by $\DS_s$.\\

So we have shown the following result.

\begin{proposition}
The family $\RS$ is birationally equivalent to the family $\DS$.
\end{proposition}

\section{Singularities of the generic fibre}
We first consider the generic fibre $\DS_{\eta}$, i.e. the fibre
over the generic point $\eta \in \A^1$.
The surface $\DS_{\eta}$ is the double cover of $\P^2$ branched over the
union of the two smooth cubic curves $B_0$, $B_1$. As a consequence,
the branch curve $B:=B_0 \cup B_1$ and hence the surface $\RS_{\eta}$, 
has singularities of type $A_{2n-1}$ at a point where $B_0$ and $B_1$ 
have intersection multiplicity $n$. A direct computation learns
that, for generic $s$, the curves $B_0$ and $B_1$ intersect in 
exactly four points $P_i,\;\; i=1,2,3,4$.
\[
\begin{array}{|c||c|c|c|c|c|}
\hline
\textup{Point}&P_1&P_2&P_3&P_4\\[1mm]
\hline
\textup{Coordinates}&(1:0:0)&(0:1:0)&(1:1:2)&(-1:1:0)\\[1mm]
\hline
\textup{Multiplicity}&3&3&2&1\\[1mm]
\hline
\textup{Singularity}&A_5&A_5&A_3&A_1\\[1mm]
\hline
\end{array}
\]

So the surface $\RS_{\eta}$ has exactly $4$ singular points $Q_i,i=1,2,3,4$
corresponding to the singular points $P_i,i=1,2,3,4$ of the branch curve $B$.
All singularities are of type $A_k$, hence we have the following corollary.

\begin{corollary}
The minimal resolution of $\DS_{\eta}$ is a K3 surface.
\end{corollary}

As is well-known (see e.g. \cite{BPHvdV}), the minimal resolution 
\[\pi: S \lra \DS_{\eta}\]
is obtained by replacing a  singular point of type $A_k$ by a
a chain of $k$ projective lines. For example, to resolve the singularity of 
the branch curve $B$ at $P_1$, one needs three blow-ups and hence we get three 
exceptional divisors 
$D_{1,2},D_{1,1},D_{1,0}$, intersecting in a chain.
The strict transform of the branch locus intersects the exceptional divisor 
coming from the last transform, say $D_{1,0}$, in two distinct points and 
does not intersect the other two. By Riemann--Hurwitz, the double cover of 
$D_{1,2}$ and $D_{1,1}$ is an unramified double cover of $\P^1$, i.e., two 
copies of $\P^1$; the double cover of $D_{1,0}$ is a double cover of $\P^1$ 
ramified above two points, i.e., it is isomorphic to $\P^1$.
Let $E_{1,i}$ and $E_{1,-i}$ be the two components of the double cover of $D_{1,i}$, 
for $i=1,2$, and let $E_{1,0}$ the cover of~$D_{1,0}$.
As $D_{1,1}$ intersect $D_{1,0}$ and $D_{1,2}$, we have that $E_{1,1}$ intersect 
$E_{1,0}$ and $E_{1,\pm 2}$.
As the name of the components are only defined up to a switch of the cover, 
we define $E_{1,2}$ as the component of the double cover  $D_{1,2}$ intersecting $E_{1,1}$. 
It follows that $E_{1,-1}$ intersects $E_{1,-2}$ and $E_{1,0}$.
Hence the divisors $E_{1,-2},E_{1,-1},E_{1,0},E_{1,1},E_{1,2}$ give rise to an $A_5$-configuration. 
We will use the analogous notations for the exceptional divisors lying above the other points $Q_i$, $i=2,3,4$. So we have exceptional divisors $E_{i,j}$ on~$S$ lying above the point $Q_i$;
the divisor $E_{i,0}$ will always denote the middle one, the divisors $E_{i,\pm 1}$ 
will be those intersecting $E_0$, for $i=1,2,3$,
the divisor $E_{i,\pm 2}$ will denote the divisors intersecting $E_{i,\pm 1}$,
for $i=1,2$.

\section{Some divisors on the generic fibre}
It is our aim to determine the Picard lattice of the smooth model $S$ of the
generic fibre~$\DS_{\eta}$. Not all the divisors we present here are defined 
over $K=\Q(s)$, some of them are defined only over an algebraic 
extension of degree $2$. The divisors we present are of three types:

\begin{itemize}
\item hyperplane section,
\item exceptional divisors coming from the resolution of the singularities of $\DS_{\eta}$,
\item pullback on $S$ of components of the pullback on $\DS_{\eta}$ of plane lines that are tritangent to the branch locus $B$.
\end{itemize}

We will denote the hyperplane class by $H$. As to the exceptional divisors, 
we have seen that the resolution of the singularities of $\DS_{\eta}$ gives rise to $5+5+3+1=14$ linearly independent divisors $E_{i,j}$.\\

Further divisors on $S$ can be obtained from the pullback of lines of
 $\P^2$ that intersect the branch locus $B$ with even multiplicity everywhere.
These pullbacks are themselves linearly equivalent to the hyperplane section, 
but they are reducible and their components are linearly independent to the 
fifteen divisors listed so far.

Let $\overline{K}$ be a fixed algebraic closure of $K=\Q (s)$,
and let $\alpha$ be a square root of $s^2-s$ inside $\overline{K}$.
We put $L:=K (\alpha)$ and consider the following lines in $\P_L^2$.
\[ 
\begin{array}{llll}
\ell_1\colon z =0,\;\;\; &\ell_2\colon z=2x,\;\;\;& \ell_3\colon z=2y,\;\;\;&\ell_4\colon z=x+y,\\[2mm]
\ell_5\colon z=\frac{x}{s+\alpha},\;\;\;&\ell_6\colon z=\frac{y}{s+\alpha},\;\;\;&\ell_7\colon z=\frac{x}{s-\alpha},\;\;\;&\ell_8\colon z=\frac{y}{s-\alpha}.\\[2mm]
\end{array}
\]

\begin{lemma}\label{l:bitangent}
Let $\ell$ be any of the lines defined above, that is, $\ell=\ell_i$ for some $i \in\{1,...,8\}$.
Then $\ell$ intersects $B$ with even multiplicity everywhere.
\end{lemma}
\begin{proof}
Notice that in the equation of any of the lines $\ell_i$, one can express $z$ in terms of $x$ or $y$.
Then, in order to prove the statement, it is enough to substitute $z$ with the expression in terms of the other variables in the equation of $B$, and check that in this way one obtains a polynomial that is a square in $\overline{L}[x,y]$.
\end{proof}

\begin{proposition}\label{p:PullBack}
Let $\ell_i$ be one of the lines defined above,
and let $D_i:=\pi^{-1}(\ell_i)$ denote the pullback of $\ell_i$ on $\DS_{\eta}$.
Then $D_i$ splits into two components, $L_i$ and $L'_i$,
both isomorphic to $\ell_i$.
\end{proposition}
\begin{proof}
It is obvious in view of Lemma~\ref{l:bitangent}.
For more details see, for example,~\cite[Proposition 1.2.26]{festi}.
\end{proof}

For every line $\ell_i$ we explicitly fix a component of the pullback on $\DS_{\eta}$, 
by writing down the explicit equations.
\[
\begin{array}{ll}
L_1 \colon z=0\; , \;\; w=2xy(x+y) & L_2 \colon z=2x\; , \;\; w=2x^2(x-y)\\[2mm]
L_3 \colon z=2y\; , \;\; w=2y^2(y-x)&L_4  \colon z=x+y\; , \;\; w=\alpha (x-y)^2(x+y)\\[2mm]
L_5 \colon z=\frac{x}{s+\alpha}\; , \;\; w=xy(x-(\alpha+s)y)/s &L_6 \colon z=\frac{y}{s+\alpha}\; , \;\; w=xy(y-(\alpha+s)x)/s\\[2mm]
L_7  \colon z=\frac{x}{s-\alpha}\; , \;\; w=xy(x+(\alpha-s)y)/s & L_8  \colon z=\frac{y}{s-\alpha}\; , \;\; w=xy(y+(\alpha-s)x)/s\\[2mm]
\end{array}
\]

\section{Intersection numbers}
These eight divisors, together with the fifteen aforementioned ones, 
give us a list of twenty-three divisors. The next step is to compute 
the intersection numbers of these divisors.

\begin{lemma}\label{l:IntH}
We have the following values for the intersection numbers of $H$:

\begin{enumerate}
\item $H^2=2$;
\item $H\cdot E_{i,j}=0$, for every $i,j$;
\item $H\cdot L_i=1$, for every $i$.
\end{enumerate}
\end{lemma}

\begin{proof}
\begin{enumerate}
\item By definition, $H$ is the pullback of a line in the plane.
Two lines in the projective plane always meet in one point.
As $\DS_{\eta}$ is a double cover of the plane, the pullbacks of the lines will 
meet in two points. As $S$ is birational to $\DS_{\eta}$ and we considered two 
general lines, the statement holds on $S$ as well. Hence, using the fact 
that lines in the plane are all linearly equivalent, we conclude that $H^2=2$.
\item One can choose a line in $\P^2$ not passing through $P_i$, for $i=1,...,4$.
The statement follows.
\item Two lines on the plane always meet in one point, 
hence the pullback of $\ell_i$ meets $H$ in two points. But this pullback splits into 
two isomorphic components, $L_i$ and $L'_i$, each of them intersecting $H$ in one point. 
\end{enumerate}
\end{proof}

\begin{lemma}\label{l:IntE}

The intersection numbers of the exceptional divisors are as follows:

\begin{enumerate}
\item $E_{i,j}^2=-2$;
\item $E_{i,j}\cdot E_{i,j'}=1$ (for $j \neq j'$) if and only if $|j-j'|=1$, 
$E_{i,j}\cdot E_{i,j'}=0$ otherwise;
\item $E_{i,j}\cdot E_{i',j'}=0$, for every $i\neq i'$.
\end{enumerate}
\end{lemma}
\begin{proof}
\begin{enumerate}
\item By construction, every $E_{i,j}$ has genus $0$.
Since $S$ is a K3 surface, its canonical divisor $K$ is trivial.
Then the result follows by the adjunction formula (cf. \cite[Proposition V.1.5]{hartshorne}).
\item Follows directly from the labeling of the components $E_{i,j}$.
\item If $i\neq i'$, it means that the divisors lie above distinct points of $\DS_{\eta}$, and therefore they do not intersect.
\end{enumerate}
\end{proof}

\begin{lemma}\label{l:IntL}
We have the following intersection matrix for $L_1,...,L_8$.
$$
\begin{pmatrix}
-2 & 0 & 0 & 0 & 0 & 0 & 0 & 0\\
0 & -2 & 0 & 0 & 0 & 0 & 0 & 0\\
0 & 0 & -2 & 0 & 0 & 0 & 0 & 0\\
0 & 0 & 0 & -2 & 0 & 1 & 0 & 1\\
0 & 0 & 0 & 0 & -2 & 0 & 1 & 0\\
0 & 0 & 0 & 1 & 0 & -2 & 0 & 1\\
0 & 0 & 0 & 0 & 1 & 0 & -2 & 0\\
0 & 0 & 0 & 1 & 0 & 1 & 0 & -2\\
\end{pmatrix}
$$
\end{lemma}
\begin{proof}
By Proposition~\ref{p:PullBack}, $L_i$ is a genus $0$ curve on the K3 surface $S$
we have $L_i^2=-2$ by the adjunction formula.
The intersection numbers $L_i\cdot L_j$ for $i\neq j$ can be explicitly computed on $\DS_{\eta}$ (see the accompanying {\tt MAGMA} code \cite{FvS18}), 
using their defining equations. Also, we already know that if $\ell_i$ and~$\ell_j$ meet 
in any of the points $P_1,...,P_4$, then $L_i\cdot L_j=0$.
\end{proof}

The intersection numbers $L_i\cdot E_{j,k}$ are slightly more difficult to determine, 
as the exceptional divisors $E_{j,k}$ only live on $S$. The following partial 
determination of these intersections will be sufficient for our purposes.\\

\begin{lemma}\label{l:IntEL}
The following statements hold.
\begin{enumerate}
\item Each of the divisors $L_1, L_7, L_8$ intersect either $E_{1,2}$ or $E_{1,-2}$;
they do not intersect the other exceptional divisors above $Q_1$.
\item Each of the divisors $L_1, L_5, L_6$ intersect either $E_{2,2}$ or $E_{2,-2}$;
they do not intersect the other exceptional divisors above $Q_2$.
\item $L_2$ intersects either $E_{2,2}$ or $E_{2,-2}$;
it does not intersect the other exceptional divisors above $Q_2$;
it does intersect either $E_{3,1}$ or $E_{3,-1}$;
it does not intersect $E_{3,0}$ nor any exceptional divisor above $Q_1$.
\item $L_3$ intersects either $E_{1,2}$ or $E_{1,-2}$;
it does not intersect the other exceptional divisors above $Q_1$;
it does intersect either $E_{3,1}$ or $E_{3,-1}$;
it does not intersect $E_{3,0}$ nor any exceptional divisor above $Q_2$.
\item $L_4$ intersect $E_{3,0}$ and $E_{4,0}$;
it does not intersect any other exceptional divisor.
\item $E_{4,0}$ intersects $L_1$ and $L_4$;
it does not intersect any other of the $L_i$'s.
\item $L_1,L_5,...,L_8$ do not intersect any exceptional divisor above $Q_3$.
\item $L_5,L_6$ do not intersect any exceptional divisor above $Q_1$.
\item $L_7,L_8$ do not intersect any exceptional divisor above $Q_2$.
\end{enumerate}
\end{lemma}
\begin{proof}

The statements can easily be proven by case by case analysis.
As an example, we provide the argument for the first statement.
The others are analogous.
 
The lines $\ell_1, \ell_7, \ell_8$ intersect at the point $P_1$.
This implies that $L_1, L_2, L_8$ intersect the exceptional divisor lying above $Q_1$.
As none of these lines is tangent to any of the irreducible components of $B$ in $P_1$, 
we have that they get separated from the branch locus already after the first blow up,
i.e., they intersect $D_{1,2}$.
The statement follows from the definition of $E_{1,\pm 2}$.
\end{proof}

\section{The geometric Picard lattice of $S$}\label{s:PicS}

We let $\overline{S}=\overline{K} \otimes_K S$ and denote by $\Lambda$  the sublattice of $\Pic \overline{S}$ generated by the divisors $H, L_i, E_{j,k}$.

\begin{proposition}\label{p:Lambda}
The lattice $\Lambda$ has rank $19$, signature $(1,18)$ and its discriminant group 
is isomorphic to $\Z/12\Z$.
\end{proposition}
\begin{proof}

Lemma~\ref{l:IntEL} determines some of the intersection numbers $L_i\cdot E_{j,k}$, but not all of them.
In fact, the intersection numbers concerning the exceptional divisors above $Q_1$ and $Q_2$ are only determined up to the double cover involution.
The same holds for the intersection numbers $L_3\cdot E_{3,\pm 1}$ and $L_2\cdot E_{3,\pm 1}$. Using the fact that we only defined these exceptional divisors up to the double cover involution, we can fix some of the intersection numbers.

We define $E_{1,2}$ as the extremal exceptional divisor above $Q_1$  intersecting~$L_1$.
It follows that $L_1\cdot E_{1,2}=1$ and  $L_1\cdot E_{1,-2}=0$.
Analogously, we define $E_{2,2}$ as the extremal exceptional divisor above $Q_2$  intersecting $L_1$.
It follows that $L_1\cdot E_{2,2}=1$ and  $L_1\cdot E_{2,-2}=0$.
Finally, we define $E_{3,1}$ as the extremal exceptional divisor above $Q_3$ intersecting $L_2$.
It hence follows that $L_2\cdot E_{3,1}=1$ and $L_1\cdot E_{3,-1}=0$.

These definitions leave fourteen intersection numbers pairwise undetermined.
The values of these intersection numbers can only vary between $1$ and $0$,
and hence it is possible to run through all the possible values. 
The Picard lattice of a K3 surface injects into the $H^{1,1}$ of the surface 
(in fact, Lefschetz $(1,1)$-theorem tells us that if $X$ is a K3 surface over $\C$, then $\Pic X \cong H^2(X,\Z)\cap H^{1,1}(X)$, cf.~\cite[1.3.3]{huybrechts}).
As the $H^{1,1}$ of a K3 surface is $20$-dimensional, 
we only look for those combinations of intersection numbers giving a Gram matrix of rank less than or equal to $20$.

Among all the combinations in $\{0,1\}^{\times 7}$, four satisfy this condition.
A computation shows that these four combinations all give rise to a lattice as 
in the statement of the result.
\end{proof}

\begin{corollary}
The lattice $\Lambda$ is isometric to the lattice
$$
U\oplus E_8(-1)^{\oplus 2} \oplus \langle -12 \rangle .
$$ 
\end{corollary}
\begin{proof}
The lattice in the statement has the same rank, signature, discriminant group, and discriminant form as $\Lambda$. Then the result follows from \cite[Corollary 1.13.3]{nikulin}.
\end{proof}

We now will show that $\Lambda$ is in fact the full Picard lattice.

\begin{lemma}\label{l:NIsotr}
The family $\DS$ is not isotrivial.
\end{lemma} 
\begin{proof}
We will see in the next section that 
the fibers above $s=1$ and $s=-1$ have non-isometric geometric Picard lattices (cf. subsections~\ref{ss:s1} and~\ref{ss:sm1}),
and therefore
the fibers themselves are not isomorphic as K3 surfaces.
\end{proof}

\begin{corollary}\label{c:PN}
$\rho (\overline{S}) \le 19$.
\end{corollary}
\begin{proof}
As the family $\DS$ is a $1$-dimensional non-isotrovial family.
Then, recalling that~$S$ is birational to the generic fiber $\DS_{\eta}$ of the family,
the statement follows for example from \cite[Corollary 3.2]{dolgachev}.
\end{proof}

\begin{proposition}\label{p:PN}
The surface $S$ has geometric Picard rank $19$.
\end{proposition}

\begin{proof}
From Corollary~\ref{c:PN} we know that the geometric Picard number is at most $19$; Proposition~\ref{p:Lambda} tells us that $\Lambda$ is a sublattice of $\Pic (\overline{S})$ of rank $19$.
Hence the rank of $\Pic (\overline{S} )$ is  exactly $19$.
\end{proof}

\begin{theorem}\label{t:Main}
The lattice $\Lambda$ is the whole geometric Picard lattice $\Pic (\overline{S})$.
\end{theorem}
\begin{proof}
In this proof, for sake of brevity, let $P$ denote the geometric Picard lattice of~$S$.
Assume by contradiction that $\Lambda$ is not the whole geometric Picard lattice $P$ of $S$;
from Proposition~\ref{p:PN} the Picard lattice of~$S$ 
(here simply denoted by  $P$) has rank $19$, 
and therefore $\Lambda$ is a proper finite-index sublattice of $P$.
As $\Lambda$ has discriminant $12$ (cf.~\ref{p:Lambda}),
it follows that $P$ has discriminant~$3$.
So $P$ is an even lattice of rank $19$ and discriminant $-3$,
getting a contradiction as even lattices of odd rank have even discriminant.
This shows that $\Lambda$ is the whole Picard lattice.
\end{proof}


\begin{corollary}\label{c:TX}
The transcendental lattice $T(S)$ of $S$ is isometric to the lattice
$$
U\oplus \langle 12 \rangle .
$$
\end{corollary}
\begin{proof}
It follows from \cite[Corollary 1.13.3]{nikulin}, 
noticing that the lattice in the statement satisfies the hypothesis.
Note that we know the form on the discriminant group $A_T$ as it is minus the form on $A_{\NS (X)}$,
which we know by Theorem~\ref{t:Main}. 
\end{proof}

\section{Special fibers}\label{s:specials}
For special values of $s$, the singularities of the branch curve $B$ change. 
We will compute the geometric Picard lattice in these particular cases,
obtaining examples of singular K3 surfaces with non-isometric Picard lattice.
This implies that these fibers are not isomorphic, 
proving Lemma~\ref{l:NIsotr}.

The computations performed and the arguments used in these special cases are 
completely analogous to the general case, and therefore they will be omitted in this paper.
They are nevertheless available in the accompanying \texttt{MAGMA} code (cf.~\cite{FvS18}).
Notice that as these surfaces turn out to have Picard lattice with maximal rank, no 
particular upper bound on the rank is needed.

\subsection{The case $s=1$}\label{ss:s1}
The surface $\DS_1$ in the weighted projective space $\P (1,1,1,3)$ is a double cover of $\P^2$ ramified along a sextic which
has exactly seven singular points: 

\begin{align*}
Q_1 = (1:0:0:0), &\; Q_2 = (0:1:0:0),\\
Q_3 = (1:1:2:0), &\; Q_4 = (-1: 1:0:0),\\
Q_5 = (1:0:1:0), &\; Q_6 = (0:1:1:0),\\
Q_7 = (1:& 1: 1:0).
\end{align*}
The points $Q_1$ and $Q_2$ are singularities of $A_5$-type; the point $Q_3$ is of $A_3$-type;
the points $Q_4,...,Q_7$ are of $A_1$-type. 

\begin{figure}[h]
\includegraphics[scale=0.25]{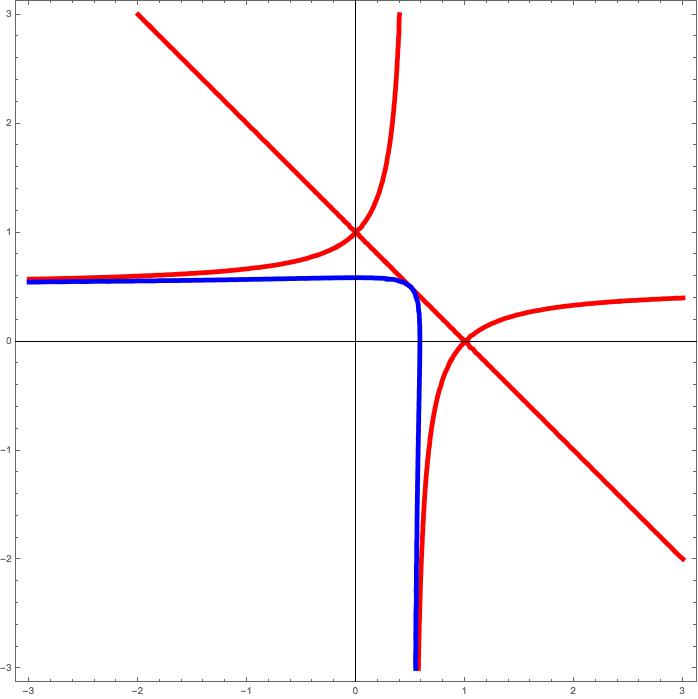}

{\em The curve $B$ for $s=1$.}
\end{figure}
Consider the following five lines of $\P^2$.
\[\ell_1\colon z =0,\;\;\;\ell_2\colon z=2x, \;\;\;\ell_3\colon z=2y,\;\;\;\ell_4\colon z=x, \;\;\; 
\ell_5\colon z=y.\]
All these lines intersect the branch locus of $\DS_1$ with even multiplicity everywhere.
Let $S_1\to \DS_1$ be the resolution of the  singularities of $\DS_1$.
Proceeding as in the previous sections, we find that
the (geometric) Picard lattice of $S_1$ is isometric to the lattice 
$$
U\oplus E_8(-1)^{\oplus 2}\oplus \langle -4\rangle \oplus \langle -2 \rangle.
$$
From this it also follows that the transcendental lattice $T(S_1)$ is isometric to the lattice $\langle 2 \rangle \oplus \langle 4 \rangle$ (as it is the only positive definite binary quadratic form with discriminant $8$).
\begin{remark}
As the discriminant of $\Pic S_1$ is $-8$, 
we know by Elkies and Sch\"{u}tt (cf.~\cite[Section 10]{schuett})
that the geometric Picard lattice of $S_1$ is realised over $\Q$,
i.e., $\Pic S_1 = \Pic \overline{S_1}$.
This can also be directly observed by noting that all the divisors used are defined over $\Q$.
\end{remark}

\subsection{The case $s=0$}\label{ss:s0}
The surface $S_0$ is isomorphic to $S_1$.
To see this, it is enough to notice that the branch locus of $S_0$ 
is obtained by reflecting the branch locus of $S_1$ with respect to the line $x+y-z=0$.
Then the two surfaces have the same Picard lattice.
It follows that the geometric Picard lattice of $S_0$ is isometric to the lattice 
$$
U\oplus E_8(-1)^{\oplus 2}\oplus \langle -4\rangle \oplus \langle -2 \rangle.
$$
Notice that the surface $S_0$ is the fiber of our family corresponding to the surface $Q$ (cf. \eqref{eq:Q}) 
we started with.
As it is a singular K3 surface, its automorphism group is infinite
and therefore it has more than just 48 automorphisms (cf. Remark~\ref{r:Symmetries}).
In~\cite[Section 10.2]{shimada}, Shimada provides a finite set of generators for the automorphisms group,
consisting of the 48 symmetries and six extra automorphisms coming from reflection with respect to the walls of the ample chamber which do not correspond to  $-2$-vectors.

\subsection{The case $s=-1$}\label{ss:sm1}
The surface $X_{-1}$ has exactly five singular points, 

\begin{align*}
Q_1 = (1:0:0:0), &\; Q_2 = (0:1:0:0),\\
Q_3 = (1:1:2:0), &\; Q_4 = (-1: 1:0:0),\\
Q_5 = (0:& 0: 1:0).
\end{align*}
The points $Q_1$ and $Q_2$ are singularities of $A_5$-type;
the point $Q_3$ is of $A_3$-type; the points $Q_4,Q_5$ are of $A_1$-type. 

\begin{figure}[h]
\includegraphics[scale=0.25]{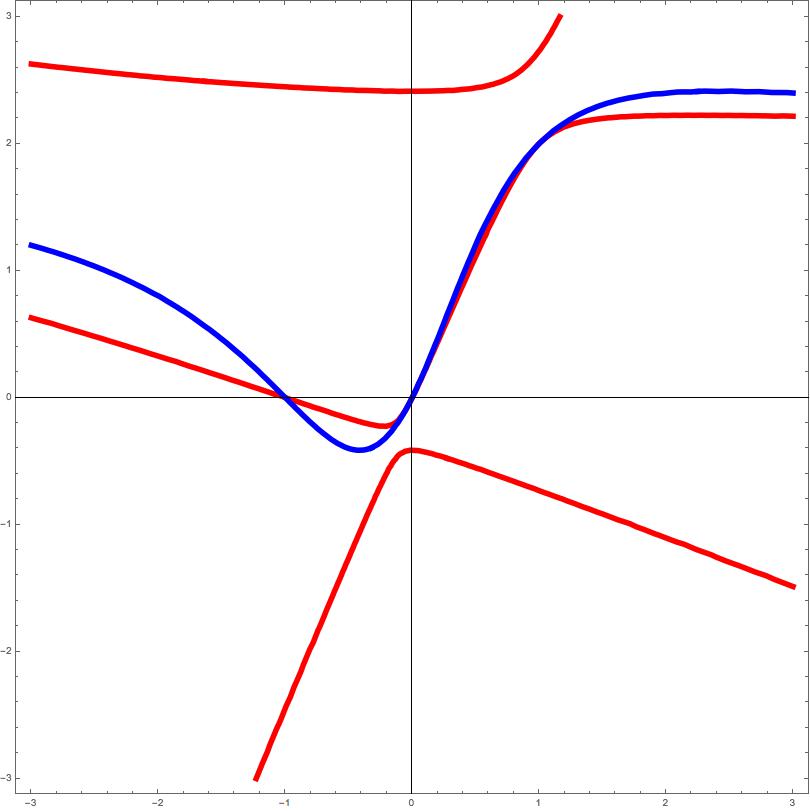}

{\em The curve $B$ for $s=-1$.}
    \label{fig:s1}
\end{figure}

In order to generate a sublattice of the Picard lattice we use the exceptional divisors lying above the singular points, 
the hyperplane sections and the divisors $L_1,...,L_8$ defined before,
with $\alpha$ a square root of $2$.
Using the same strategy as before we have that
the geometric Picard lattice of $X_{-1}$ is isometric to the lattice
$$
U\oplus E_8(-1)^{\oplus 2}\oplus \langle -12 \rangle \oplus \langle -2 \rangle.
$$
Reasoning as in Subsection~\ref{ss:s1}, 
we show that $T(S_{-1})$ is isometric to  
$\langle 2 \rangle \oplus \langle 12 \rangle$.
In~\cite[Theorem 1.2(1)]{shimada1}, Shimada proves that the automorphism group of the surface is generated by 15 involutions. 

\subsection{The case $s=2$}\label{ss:s2}
Using the same argument as in Subsection~\ref{ss:s0},
one sees that $S_{2}$ is isomorphic to $S_{-1}$.

\section{Transformation to the Ap\'ery--Fermi-family}
We have seen that 
\[\RS\colon 0=1+s+\frac{1}{u^2-1}+\frac{1}{v^2-1}+\frac{1}{w^2-1}\]
is birationally a pencil of K3 surfaces with $U \oplus \langle 12 \rangle$
as transcendental lattice.\\ 

We will now explain how this family is related to a famous piece of 
mathematics that started with Ap\'ery's spectacular proof
of the irrationality of $\zeta(3)$, in which the Ap\'ery numbers
\[ A_0=1,\;A_1=5,\;A_2=73,\;A_4=1445,\ldots, A_n:=\sum_{k=0}^n {n \choose k}^2{n+k \choose k}^2\]
played an important role. In the paper \cite{beukerspeters}, 
F. Beukers and C. Peters showed that the generating function 
\[\Phi(\lambda) :=\sum_{k=0}^\infty A_n \lambda^n\]
of the Ap\'ery numbers is a period function of a certain pencil 
$\AS \to \P^1$ of K3 surfaces. The differential operator 
\[\theta^3-\lambda(2\theta+1)(17\theta^2+17\theta+5)+\lambda^2(\theta+1)^3 ,\]
of which $\Phi(\lambda)$ is the  unique solution, holomorphic around $0$, was identified as the  Picard--Fuchs operator of the pencil, which has
\[ 0,\;\;\;17+12\sqrt{2},\;\;\;17-12\sqrt{2},\;\;\;\infty\]
as singularities. 

Somewhat later, C. Peters and J. Stienstra~\cite{petersstienstra} studied the
level sets of the Laurent polynomial 
\[F:=x+\frac{1}{x}+y+\frac{1}{y}+z+\frac{1}{z} ,\]
which turned up in the theory of {\em Fermi surfaces}, \cite{giesekerknoerrertrubowitz}. These are affine parts of K3 surfaces, and the
pencil
\[\ZS:  x+\frac{1}{x}+y+\frac{1}{y}+z+\frac{1}{z}=\xi+\frac{1}{\xi}\]
over the $\xi$-line has 
\[\theta^3 -\xi^2(\theta+1)(17\theta^2+34\theta+20)+\xi^4(\theta+2)^3,\;\;\theta=\xi\frac{d}{d\xi} \]
as Picard--Fuchs operator, and
\[0,\;\;\;\pm 3 \pm \sqrt{2},\;\;\;\infty\]
as singularities. This operator is just the {pullback of the
Ap\'ery operator} via the substitution $\lambda=\xi^2$.
In fact, the family $\AS \to \P^1$ can be identified as the {\em quotient} of 
$ \ZS \to \P^1$ by the map induced by
\[ (x,y,z,\xi) \mapsto (-x,-y,-z,-\xi)\]
and thus provides a simpler and more symmetric description of the Ap\'ery--Fermi pencil $\AS \to \P^1$
of \cite{beukerspeters}.
Obviously, the Ap\'ery--Fermi family itself is a pullback from the family
\[ \FS: x+\frac{1}{x}+y+\frac{1}{y}+z+\frac{1}{z}=t\]
over the $t$-line. Note that this family still has a 
symmetry, as the fibre over $t$ is isomorphic to the fibre over $-t$
via the map $(x,y,z) \to (-x,-y,-z)$. As a result, the Picard--Fuchs
operator  of $\FS$ has a symmetry and is the pullback via the substitution
$\mu=1/t^2$ of the operator

\[\theta^3-2\mu (2\theta+1)(10\theta^2+10\theta+3)+\mu^2(2\theta+1)(\theta+1)(2\theta+3),\;\;\;\theta:=\mu\frac{d}{d\mu},\]
with singularities at $0,\; 1/4,\; 1/36,\; \infty$. So this is a close relative of 
the Ap\'ery operator; the holomorphic solution $\Psi(\mu)$ around $0$ expands as follows:
\[\Psi(\mu)=\sum_{n=0}^\infty b_n \mu^n=1+6 t+ 90 t^2+1860t^3+44730t^4+\ldots\]
\[b_n:={2 n \choose n}a_n,\;\;\;a_n:=\sum_{k=0}^n { m \choose k}^2 {2k \choose k} .\]
The differential equation and the numbers $b_n$ appeared in \cite{domb} in certain
models of conductivity in crystals. The number $b_n$ counts the number of paths of lenght 
$2n$ in the standard cubical lattice $\Z^3$ with the Euclidean distance.
The series 
\[\sum_{n=0}^{\infty}a_n \mu^n\]
is the period of a rational elliptic surface with
singular fibres  $I_6,I_3,I_2,I_1$
and appears in \cite{stienstrabeukers}. It can be realised by the family
of cubics
\[ (x+y)(y+z)(z+x)+\mu xyz=0.\]

In their paper~\cite{petersstienstra}, Peters and Stienstra also determined  
the Picard lattice of the general fibre of the Ap\'ery--Fermi family and found 
them to be
\[ U \oplus (-E_8)^{\oplus 2} \oplus \langle -12 \rangle \]
which is exactly the lattice we found for our family of K3 surfaces!
By the global Torelli theorem for K3 surfaces, 
we suspect the two families to be related. 
One can use the determination of the Picard lattice for the 
special members of  Section~\ref{s:specials} and compare them with the 
Picard lattices of for the special members of the Ap\'ery--Fermi family determined 
in \cite{petersstienstra}. This suggests the identification 
\[ t=\pm ( 2-4s) \]
of the fibres of the families $\RS$ and $\FS$. And indeed, if we put 
\[G(x,y,z):=\frac{1}{x^2-1}+\frac{1}{y^2-1}+\frac{1}{z^2-1}\]
then the substitution
\[x \mapsto \frac{1+x}{1-x},\;\; y \mapsto \frac{1+y}{1-y},\;\;z \mapsto \frac{1+z}{1-z}\]
converts $G$ into $F$, up to a factor and a constant. To be precise, one has the following\\

\centerline{\bf \em Remarkable Identity}

\[ 4 G\left( \frac{1+x}{1-x},\frac{1+y}{1-y},\frac{1+z}{1-z}\right) =F(x,y,z)-6.\]

Hence, the above substitution maps the pencil of surfaces
\[0=1+s+\frac{1}{x^2-1}+\frac{1}{y^2-1}+\frac{1}{z^2-1}\]
to the pencil of surfaces
\[x+1/x+y+1/y+z+1/z=2-4s.\]

So we see that our family $\RS$ is, up to a linear transformation in the
parameter, the Ap\'ery--Fermi family $\FS$!
There is much more to be said about the rich algebraic geometry around 
this pencil. 
Note that by removing
denominators we obtain the family of quartics given by
$$
x^2yz+yz+xy^2z+xz+xyz^2+xy-(2-4s)xyz=0.
$$

We refer to \cite{bertinlecacheux} for a recent contribution.
In that paper, Bertin and Lecacheux prove (among other things) that the generic member of the family admits 27 elliptic fibrations, and they also describe the singular fibers of all of them.\\

We remark that the Ap\'ery--Fermi family and its Picard--Fuchs operator also
play a role in the analysis of the equal mass scalar Feynmann diagram shown below.\\

\begin{center}

\includegraphics[width=10cm]{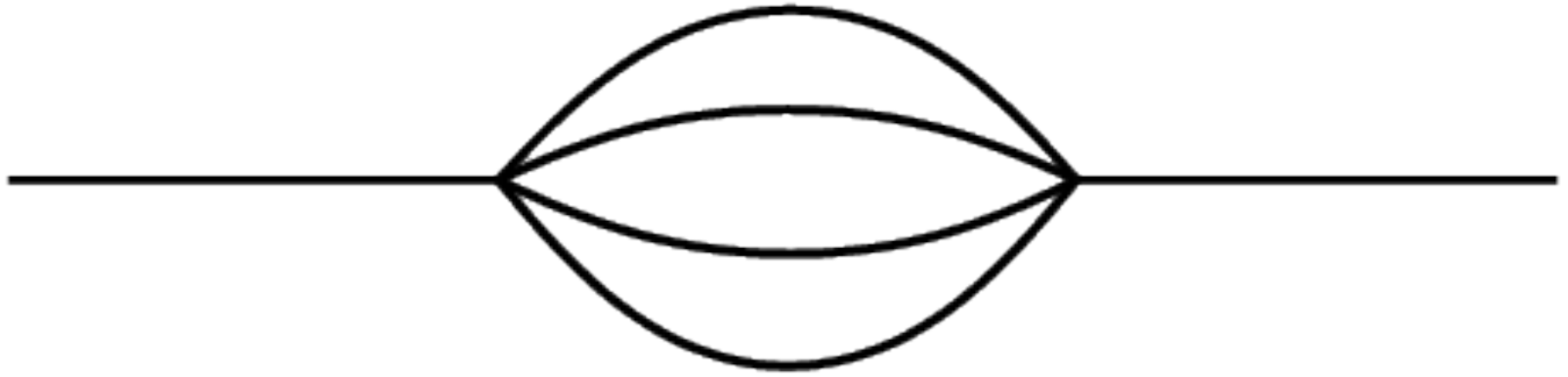}\\

{\em Banana graph.}
\end{center}
 
From the extensive literature on the subject, we mention \cite{verrill},
\cite{vanhove}, \cite{blochkerrvanhove}, \cite{borweinstraubwan}. \\

Another interesting occurence of the family of K3 surfaces with the same Picard lattice is the family of {\em Hessian quartics} of the family of cubic surfaces 
\[C_a\colon 0=x_0^3+x_1^3+x_2^3+x_3^3-3x_0^2(x_1+x_2+x_3)\]
described in~\cite[Proposition 5.7.(4)]{dardanellivangeemen}. 
This gives yet another realisation of the pencil
 as a family of quartics in $\P^3$.\\

We conlude with the remark that the appearance of the irrationality $Q$ from the introduction 
is to be expected from the symmetry considerations. 
In fact, as already stated in \cite{hennsmirnov}, it is natural to introduce 
a variable $z$ related to the third Mandelstam variable $u$ by the relation
\[-u =m^2\frac{(1-z)^2}{z}.\]
Then the well-known relation 
\[ s+t+u=4m^2\] 
between the Mandelstam variables translates directly into the algebraic relation 
\[ \frac{(1-x)^2}{x}+\frac{(1-y)^2}{y}+\frac{(1-z)^2}{z}+4=0,\]
which upon expansion is just the equation
\[ x+\frac{1}{x}+y+\frac{1}{y}+z+\frac{1}{z}=2 .\]
By the above {\em Remarkable Identity}, this surface is the transform of the surface $R=\RS_0$.

On this surface there is a symmetric collection of divisors and the $2$-loop amplitudes appearing
in Bhabha scattering seem to be special solutions of {\em a natural differential system on this K3 surface}, rather then on the affine $x,y$-plane, and a theory of harmonic polylogarithms should properly be formulated as living on the above K3 surface.\\
 
So we see the `same' pencil of K3 surfaces appearing over and over again, 
with a multitude of links to various physical contexts. 
A question posed 
by J. Stienstra \cite{stienstra} is if there are more direct links 
between these physical occurrences that would provide {\em a priori} explanations for the miracles.

\section*{Acknowledgements}
We would like to thank Johannes Henn for asking the original question that led 
to this paper. We also thank Marco Besier, Claude Duhr,  Alice Garbagnati, Bert van Geemen, 
Davide Cesare Veniani, Stefan Weinzierl and the anonymous referee for helpful comments and discussions. 
The first author was supported by SFB/TRR 45.
The authors also would like to express a special thanks to the Mainz Institute for Theoretical Physics (MITP) for its hospitality and support,
and to Oliver Labs for his software for visualization of algebraic surfaces, ``Surfex''.


\begin{thebibliography}{99}

\bibitem{BPHvdV} W. Barth, C. Peters, K. Hulek,  A. van de Ven, 
{\em Compact Complex Surfaces}, Springer Verlag, (2004).

\bibitem{bertinlecacheux} M. Bertin, O. Lecacheux, 
{\em Ap\'ery-Fermi pencil of $K3$-surfaces and their $2$-isogenies},
{\tt math.AG 1804.04394} (2018).

\bibitem{beukerspeters} F. Beukers, C. Peters, {\em A family of K3 surfaces and $\zeta(3)$}, Journal f\"ur reine und angewandte Mathematik {\bf 351}, (1984), 42 - 54.

\bibitem{bhabha} H. J. Bhabha, {\em The scattering of positrons by electrons with exchange on Dirac’s theory of the positron}, Proceedings of the Royal Society,  
Volume {\bf 154}, issue 881, (1936), 195-206.

\bibitem{blochkerrvanhove} S. Bloch, M. Kerr, P. Vanhove, {\em  Local mirror symmetry and the sunset Feynman integral}, {\tt arXiv:hept-th:1601.08181} (2016).

\bibitem{boncianietall} R. Bonciani, A. Ferroglia, M. Mastrolia, E. Remiddi, J. van der Bij, {\em Planar box diagram for the $(N(F)=1)$ two loop QED virtual corrections to Bhabha scattering},
Nucl. Phys. {B681},(2004), 261-291.

\bibitem{boncianietall2} R. Bonciani, A. Ferroglia, M. Mastrolia, E. Remiddi, J. van der Bij, {\em Two-loop  $N(F)=1$ QED Bhabha scattering differential cross section}, Nucl. Phys. {B701}, (2004), 121-179.


\bibitem{borweinstraubwan} J. Borwein, A. Straub and J. Wan, {\em Three-Step and Four-Step Random Walk Integrals}, Exper. Math., 22 (2013), 1-14.

\bibitem{bddpt} J. Broedel, C. Duhr, F. Dulat, B. Penante and L. Tancredi,
{\em Elliptic polylogarithms and Bhabha scattering at two loops},
to appear.


\bibitem{bvsw} M. Besier, D. van Straten, S. Weinzierl, {\em  Rationalizing roots: an algorithmic approach}, {\tt arXiv:1809.10983} (2018).
 
\bibitem{conwaysloane} J. Conway, N. Sloane, {\em Sphere packings, lattices and groups}, Grundlehren der Mathematischen Wissenschaften {\bf 290},
(third edition, with additional contributions by E. Bannai, R. E. Borcherds,
J. Leech, S. P. Norton, A. M. Odlyzko, R. A. Parker, L. Queen and B. B. Venkov),
Springer Verlag (1999).

\bibitem{dardanellivangeemen} E. Dardanelli, B. van Geemen, 
{\em  Hessians and the moduli space of cubic surfaces}, in: Algebraic Geometry,
Contemp. Math. {\bf 422}, 17 -36, AMS, Providence, RI (2007).

\bibitem{dolgachev} I. Dolgachev, {\em Mirror Symmetry for lattice polarized K3 surface}, J. Math. Sci {\bf 81},(1996), 2599 - 2630.

\bibitem{domb} C. Domb, {\em On the theory of cooperative phenomena in crystals}, Advances in Phys. {\bf 9}, 149-361, (1960).

\bibitem{festi} D. Festi,{\em Topics in the arithmetic of del Pezzo and K3 surfaces}, PhD-thesis, Leiden (2016).

\bibitem{FvS18} D. Festi, D. van Straten, {\em Accompanying \texttt{MAGMA} code}, available at {\tt www.staff.uni-mainz.de/dfesti/BhabhaCode.txt}, (2018).

\bibitem{giesekerknoerrertrubowitz} D. Gieseker, H. Kn\"orrer, E. Trubowitz, 
{\em An overview of the geometry of algebraic Fermi curves}, in: 
Algebraic geometry: Sundance 1988, 19 - 46, Contemp. Math., {\bf 116}, 
Amer. Math. Soc., Providence, RI, 1991.

\bibitem{hartshorne} R. Harsthorne, {\em Algebraic Geometry}, Graduate Texts in Mathematics {\bf 52}, Springer Verlag (1977).

\bibitem{hennsmirnov} J. Henn, V. Smirnov, {\em Analytic results for the two-loop master integrals for Bhabha scattering I}, Journal of High Energy Physics, (2013). 

\bibitem{huybrechts} D. Huybrechts, {\em Lectures on K3 surfaces}, Cambridge Studies in advanced Mathematics {\bf 158}, Cambridge University Press (2016).

\bibitem{nikulin} V. Nikulin, {\em Integer symmetric bilinear forms and some os their geometric applications}, Izv. Akad. Nauk SSSR Ser. Mat.{\bf 43},(1979), 111-177.

\bibitem{remiddivermaseren} E. Remiddi, J. Vermaseren, {\em Harmonic Polylogarithms}, Int. J. Mod. Phys. A {\bf 15} (2000), 725-754.

\bibitem{petersstienstra} C. Peters, J. Stienstra, {\em A pencil of K3- surface related to Ap\'ery's recurrence for $\zeta(3)$ and Fermi surfaces for potential zero}, in: Arithmetic of complex manifolds (Erlangen, $1988$), Lecture Notes in Mathematics {\bf 1399} (W. Barth, H. Lange, eds.), Springer Verlag (1989).

\bibitem{shimada} I. Shimada, {\em An algorithm to compute automorphism groups of {$K3$} surfaces
              and an application to singular {$K3$} surfaces}, 
              Int. Math. Res. Not. IMRN, {\bf 22} (2015), 11961--12014.

\bibitem{shimada1} I. Shimada, {\em The automorphism groups of certain singular K3 surfaces and an Enriques surface},
preprint.

\bibitem{schuett} M. Sch\"{u}tt, {\em K3  surfaces with Picard rank 20}, 
Algebra Number Theory {\bf 4} (2010), no. 3, 335-356.

\bibitem{stienstra} J. Stienstra, {\em private communication}, HIM Bonn, Februari $2018$.

\bibitem{stienstrabeukers} J. Stienstra, F. Beukers, {\em On the Picard-Fuchs Equation and the Formal Brauer group of certain Elliptic K3-surfaces}, Math. Ann, {\bf 27} (1985), 269 -304.

\bibitem{vanhove} P. Vanhove, {\em Feynman integrals, toric geometry and mirror symmetry}, {\tt arXive 1807.11466}.

\bibitem{verrill} A. Verrill, {\em The L-series of certain rigid Calabi-Yau Threefolds}, J. Number Theory {\bf 81} (2000), 310 - 334.
\end{thebibliography}
\end{document}